\documentclass[12pt]{amsart}
\usepackage{fullpage} 

\usepackage{amsmath, amsthm, amssymb}
\usepackage[normalem]{ulem}
\usepackage[dvipsnames]{xcolor}
\usepackage{epsfig}
\usepackage{enumitem,subcaption}
\setenumerate{label=\arabic*.,ref=\arabic*}
\usepackage{setspace}
\usepackage{comment}
\usepackage[all]{xy}
\usepackage{hyperref}
\usepackage{setspace}
\usepackage[colorinlistoftodos]{todonotes}

\newtheorem{thm}{Theorem}[section]

\newtheorem{cor}[thm]{Corollary}

\newtheorem{prop}[thm]{Proposition}

\newtheorem{definition}[thm]{Definition}
\newtheorem{remark}[thm]{Remark}
\newtheorem{ex}[thm]{Example}
\newtheorem{nota}[thm]{Notation}
\newtheorem{notarem}[thm]{Notation and Remarks}
\newtheorem{Main Results}[thm]{Main Results}
\newtheorem{Result}[thm]{Result}

\DeclareMathOperator{\tr}{tr}
\definecolor{dgreen}{RGB}{0,128,0}

\newcommand{\orb}{\mathcal O}

\newcommand{\R}{\mathbf R}
\newcommand{\Z}{\mathbf Z}

\newcommand{\bs}{\backslash}

\newcommand{\vol}{{\rm vol}}

\newcommand{\iso}{\iiso}

\newcommand{\restr}[1]{\lower0.4ex\hbox{$|$}\lower0.7ex
	\hbox{$\scriptstyle{#1}$}}
\definecolor{olive}{RGB}{128,128,0}

\definecolor{orange}{rgb}{1,0.5,0}

\definecolor{darkpastelpurple}{rgb}{0.59,0.44,0.84}

\newcommand{\cc}[1][U]{(\widetilde{#1},G_{#1}, \pi_{#1})}
\newcommand{\wtu}{\widetilde{U}}

\newcommand{\iiso}{\operatorname{Iso}}
\newcommand{\isomax}{\iiso^{\rm max}}

\newcommand{\coloneqq}{:=}

\newcommand{\End}{\operatorname{End}}

  \def\R{\mathbb R} \def\Z{\mathbb{Z}}

\def\O{{\mathcal{O}}}

\def\bs{{\backslash}}

\def\Iso{{\operatorname{{Iso}}}}

\newcommand{\codim}{\operatorname{codim}}
\def\Id{{\operatorname{{Id}}}}
\def\p{(p)}
\makeatletter
\@namedef{subjclassname@2020}{\textup{2020} Mathematics Subject Classification}
\makeatother

\begin{document}

\title{Do the Hodge spectra distinguish orbifolds from manifolds?\\ Part 2}

\author{Katie Gittins}
\address{Department of Mathematical Sciences, Durham University,
Mathematical Sciences \& Computer Science Building,
Upper Mountjoy Campus, Stockton Road,
Durham DH1 3LE,
United Kingdom.}
\email{katie.gittins@durham.ac.uk}

\author{Carolyn Gordon}
\address{Department of Mathematics, Dartmouth College, Hanover, NH, 03755, USA.}
\email{csgordon@dartmouth.edu}

\author{Ingrid Membrillo Solis}
\address{Mathematical Sciences, University of Southampton,  Southampton SO17~1BJ, United Kingdom}
\email{i.membrillo-solis@soton.ac.uk}

\author{Juan Pablo Rossetti}
\address{Universidad Nacional de Cordoba, Medina Allende s/n, Ciudad Universitaria, 5000 Cordoba, Argentina. 
}
\email{jprossetti@unc.edu.ar}

\author{Mary Sandoval}
\address{Department of Mathematics, Trinity College, 300 Summit Street, Hartford, CT, 06106, USA.}
\email{mary.sandoval@trincoll.edu}

\author{Elizabeth Stanhope}
\address{Department of Mathematical Sciences, Lewis \& Clark College, Portland, OR, 97219, USA.}
\email{stanhope@lclark.edu}


\maketitle

\begin{abstract}
In \cite{GGKM-SSS} we examined the relationship between the singular set of a compact Riemannian orbifold and the spectrum of the Hodge Laplacian on $p$-forms by computing the heat invariants associated to the $p$-spectrum. We showed that the heat invariants of the $0$-spectrum together with those of the $1$-spectrum for the corresponding Hodge Laplacians are sufficient to distinguish orbifolds from manifolds as long as the singular sets have codimension $\le 3.$ This is enough to distinguish orbifolds from manifolds for dimension $\le 3.$ Here we give both positive and negative inverse spectral results for the individual $p$-spectra considered separately. For example, we give conditions on the codimension of the singular set which guarantee that the volume of the singular set is determined, and in many cases we show by providing counterexamples that the conditions are sharp.

\end{abstract}

\section{Introduction}

A (Riemannian) orbifold is a generalization of a (Riemannian) manifold that  permits  the  presence of mild singularities, in the following sense: orbifolds of dimension $d$ are locally modeled by the orbit spaces of finite effective group actions on $\R^d$.  Orbifold singularities correspond to orbits with non-trivial isotropy and thus each singularity has an associated ``isotropy type.'' Orbifolds appear in a variety of mathematical areas and have applications in physics, in particular to string theory. In this paper, we will always assume all orbifolds under consideration are connected and compact without boundary.

The notions of the Laplace-Beltrami operator and the Hodge Laplacian extend to Riemannian orbifolds so it is natural to extend many questions from the spectral theory of manifolds to orbifolds. This paper will be concerned with inverse spectral problems for the spectrum of the Hodge Laplacian on $p$-forms, which we refer to as the $p$-spectrum. We will say that two closed Riemannian orbifolds are $p$-\emph{isospectral} if their Hodge Laplacians acting on $p$-forms are isospectral.  In particular, $0$-isospectrality means that the Laplace-Beltrami operators are isospectral.
%

 We will focus on both negative results--what properties of orbifolds cannot be detected by the $p$-spectra, and positive results--what properties of orbifolds or manifolds can be heard by the $p$-spectra for various $p$. Because  the possible presence of singular  points is a defining  characteristic  of  the class of  orbifolds,  we will be particularly interested in the question ``Can one hear the singularities of an orbifold in the $p$-spectrum?"  More precisely, one asks for various values of $p$:

\begin{enumerate}
\item \label{it:orbimani} Does the $p$-spectrum distinguish orbifolds with singularities from manifolds? If so, does the $p$-spectrum detect the topology and geometry of the set of singular points, including the isotropy types of singularities?
\item What other geometric or topological properties does the $p$-spectrum detect or fail to detect for orbifolds?
\end{enumerate}

Many authors have addressed these questions for the spectrum of the Laplace-Beltrami operator (see, for example, the literature review in \cite{GGKM-SSS} and corresponding references).   However, the question of whether the $0$-spectrum 
\emph{always} distinguishes Riemannian orbifolds with singularities from Riemannian manifolds remains open.   In the first part of this project, \cite{GGKM-SSS}, we showed that the heat invariants for the $0$-spectrum and the $1$-spectrum together can distinguish orbifolds from manifolds when the codimension of the singular set is less than or equal to three. In this paper, we will consider the inverse spectral problem for individual $p$-spectra, obtaining both positive and negative results.

\subsection{Main Results}

We summarize our main results below. 
As stating the results in full detail requires us to introduce various definitions and notation, we give an overview here and refer to the complete statements of the results (via citations in parentheses) that appear later in the paper.

First recall that orbifolds admit a natural stratification. In a connected orbifold, the collection of regular points forms a single stratum of full dimension, while the singular set is a union of lower-dimensional strata.  By the \emph{dimension} $m$ of the singular set, we mean the maximum dimension of the singular strata.    The volume of the singular set is then understood to be its $m$-dimensional Hausdorff measure.   

\subsubsection{Results for arbitrary $p$-spectra.}
Our results refer to the combinatorial Krawtchouk polynomials $K^d_p$; see subsection~\ref{sec.computation} for the definition. These polynomials depend on a positive integer parameter $d$ and a parameter $p$ lying in $\{0,\dots,d\}$.  
There is a large literature on the zeros of these polynomials (e.g., see \cite{CS90,KL96}).

The first item in Result~\ref{gen p} below is motivated by and applies the work of Roberto Miatello and Juan Pablo Rossetti \cite{MR01}, where the Krawtchouk polynomials were used to construct $p$-isospectral Bieberbach manifolds. 
\begin{Result}\label{gen p} Let ${\mathcal Orb}^d_k$ denote the class of all closed $d$-dimensional Riemannian orbifolds with singular set of codimension $k$.   
\begin{enumerate}
\item If $k\in \{1,\dots, d-1\}$ is a zero of the Krawtchouk polynomial $K^d_p$, then there exists a $d$-dimensional flat orbifold $\O\in {\mathcal Orb}^d_k$ and a Bieberbach manifold $M$ such that $\O$ is $p$-isospectral to $M$. Moreover there exist families of mutually isospectral flat orbifolds and manifolds some of which are orientable and others not. (See Proposition~\ref{it:flat-involution}.)
 \item In contrast, if $k$ is odd and $K^d_p(k)\neq 0$, then the $p$-spectrum determines the volume of the singular set of each orbifold in ${\mathcal Orb}^d_k$.   In particular, the $p$-spectrum distinguishes orbifolds in ${\mathcal Orb}^d_k$ from closed Riemannian manifolds. (See Theorem~\ref{thm:krawt4}.)
    \end{enumerate}  
\end{Result}

 The first item in Result~\ref{gen p} yields many examples of Riemannian orbifolds that are $p$-isospectral to manifolds for various $p\neq 0$, including examples that are simultaneously $p$-isospectral for all odd values of $p$.  Moreover, the codimension of the singular sets in the various examples can have either even or odd parity.  To our knowledge, the only previously known examples \cite{GR03} of isospectralities between orbifolds with singularities and manifolds were in the special case that $d$ is even and $p=\frac{d}{2}$. See the discussion~\ref{subsub middle} below concerning isospectrality in the middle degree.

Consider the case of orbifolds $\O$ with singular set of codimension one.     Result~\ref{gen p}, together with an examination of the zeros of the Krawtchouk polynomials, says in this case that the $p$-spectrum determines the volume of the singular set unless the dimension $d$ of $\O$ is even and $p=\frac{d}{2}$.   Specializing a little further, suppose that \emph{all} the singular strata have codimension one.  The underlying space of $\O$ is then a Riemannian manifold $M$ with boundary; the singular set of $\O$ (which consists of what are called ``reflector edges'') corresponds to the manifold boundary.  Moreover, the $p$-spectrum of the orbifold coincides with the $p$-spectrum of $M$ with absolute boundary conditions (Neumann boundary conditions if $p=0$.)   (See, e.g., Remark~3.16 of \cite{GGKM-SSS}.)  Result~\ref{gen p} coincides in this special case with the well-known consequence of the heat invariants stating that, except when the dimension $d$ is even and $p=\frac{d}{2}$, the $p$-spectrum of a compact Riemannian manifold determines the volume of the boundary.

For $p=0$, the Krawtchouk polynomial has no zeros, so Result~\ref{gen p} says that if the singular set has odd codimension, its volume is determined by the $0$-spectrum, i.e., by the spectrum of the Laplacian on functions.   This fact can also be seen as an immediate consequence of the heat invariants for the 0-spectrum of orbifolds given in \cite{DGGW08}.    Moreover, \cite[Theorem 5.1]{DGGW08} says that an orbifold that contains at least one singular stratum of odd codimension cannot be 0-isospectral to a manifold, even if the full singular set has even codimension.  

 Our proof of the second item in Result~\ref{gen p} uses a new spectral invariant for the $p$-spectrum of orbifolds that we introduce in Notation~\ref{parity} and Proposition~\ref{prop: odd par}.    A necessary (but not sufficient) condition for an orbifold to be $p$-isospsectral to a manifold is that this invariant vanishes.     The expression for the invariant is somewhat simpler to work with when $p=1$ than for higher $p$ and enables us to obtain weak -- but in some cases sharp -- analogues for 1-forms of the result \cite[Theorem 5.1]{DGGW08} described in the previous paragraph.   We discuss these and other results for the 1-spectrum next.

\subsubsection{Results for the $1$-spectrum.}

We obtain the following positive results for the  $1$-spectrum for orbifolds of arbitrary dimension $d$ (see Theorem~\ref{thm:halfdim}.)   We also show that some of the results below are sharp.

\begin{Result}\label{R5}
Let $\orb$ be a closed Riemannian orbifold of dimension $d$.
\begin{enumerate}

\item\label{DGGWp}  If $\orb$ contains at least one primary singular stratum of odd codimension $k<\frac{d}{2}$, then $\orb$ cannot be 1-isospectral to any closed Riemannian manifold.

\item\label{GRSp} If $\orb$ contains at least one singular stratum of codimension $k<\frac{d}{2}$, then $\O$ cannot be $1$-isospectral to any closed Riemannian manifold
$M$ such that $M$ and $\O$ have isometric infinite homogeneous Riemannian covers.  If in addition $\orb$ is good, then $\O$ cannot be $1$-isospectral to any closed Riemannian manifold $M$ such that $M$ and $\O$ have finite $1$-isospectral Riemannian covers.

\item 
For $d$ even, both statements remain true when $k=\frac{d}{2}$ provided at least one stratum of codimension $k$ has isotropy group of order at least three.

\item 
An element of ${\mathcal Orb}^d_k$ with $k$ odd can be 1-isospectral to a Riemannian manifold only if $d$ is even and $k=\frac{d}{2}$.

\end{enumerate}

\end{Result}

The conditions imposed on the codimension in the first two items of  Result~\ref{R5} are sharp at least for all even-dimensional orbifolds, and the assumption on the isotropy order in the third item cannot be removed. Indeed, when $d$ is even, the first item in Result~\ref{gen p} allows one to construct a $d$-dimensional  orbifold $\O$ with singular set of codimension $\frac{d}{2}$ and a manifold $M$ such that $\O$ and $M$ are $p$-isospectral for all odd $p$, in particular for $p=1$.  (See Example \ref{cor.ornon} and Remark~\ref{rm:zeros_of_K}, part 5.)

Item \ref{DGGWp} of Result~\ref{R5} is the weak analogue for the $1$-spectrum of \cite[Theorem 5.1]{DGGW08}, and
Item \ref{GRSp} of Result~\ref{R5} gives weak analogues for $1$-forms of results for functions from \cite{GR03} and its errata \cite{GR21} and \cite[Theorem 1.2]{S10}.

\subsubsection{Results for the middle-degree spectrum}\label{subsub middle}

As discussed in \cite{GR03} and its errata \cite{GR21},  the middle degree $p=\frac{d}{2}$ Hodge spectrum of even-dimensional Riemannian manifolds, or more generally orbifolds, contains less information than the other spectra due to Hodge duality.  This weakness of the middle-degree spectrum is also apparent in Result~\ref{gen p}: when $p=\frac{d}{2}$, the Krawtchouk polynomial $K_p^d(k)=0$ for every odd integer $k$ in the interval $[1,d-1]$; Krawtchouk polynomials for other values of $p$ have far fewer integer zeros.

\begin{Result}\label{R3}
    We exhibit a family of five $1$-isospectral flat 2-orbifolds (none of which are $0$-isospectral) that have different types of singularities. (See Example \ref{ex:isospecorb}.)

\begin{enumerate}
\item One can be viewed as a Euclidean square with absolute boundary conditions.
\item The underlying spaces of the various orbifolds include a sphere, a projective plane, and disks.
\item Some have reflector edges of various lengths while another has only cone points. Thus this family contains locally orientable and non-locally orientable orbifolds, in contrast to a result from \cite{RS20} showing that this cannot happen with $0$-isospectral orbifolds. (See Remark~\ref{rm:orbori} for the notion of local orientability of orbifolds.)
\end{enumerate}
\end{Result}

Note that a pair of 1-isospectral 2-orbifolds, one locally orientable and one not, were already constructed in \cite{GR03}. See Example~\ref{ex:d/2} for details. Also, in contrast to these five 1-isospectral 2-orbifolds which are not 0-isospectral, in \cite{RSW08} the authors give examples of $3$-orbifolds that are $0$-isospectral and are not $1$-isospectral (see Remark~\ref{rm:must-be-odd}). 

\subsection{Plan of the paper.}
This paper is organized as follows: Section 2 recalls the relevant definitions, notation, and results from \cite{GGKM-SSS}, and the Krawtchouk polynomials that will be used throughout. Section 3 contains proofs of all the negative inverse spectral results; Section 4 contains proofs of all the positive inverse spectral results.

\section{Preliminaries}\label{sec:newbackground}
We begin this section by recalling the definition and some basic properties of orbifolds.  See \cite{GGKM-SSS} or the references therein for more detail. We will then review results from \cite{GGKM-SSS} addressing heat asymptotics for the Hodge Laplacian on $p$-forms for orbifolds. In the final subsection, we will recall the definition and some properties of Krawtchouk polynomials and explore their link to the heat invariants.    

\subsection{Notation}\label{notationpartone}
Here, we define the notation needed for this paper from \cite{GGKM-SSS}. 

\begin{definition}\label{def:orbchart}
For a connected open subset $U \subseteq X$, an \emph{orbifold chart} (of dimension $d$) over $U$ is a triple $\cc$ where $\wtu \subseteq \R^d$ is a connected open subset, $G_U$ is a finite group acting on $\wtu$ effectively and by diffeomorphisms, and $\pi_U\colon \wtu \to X$ is a map inducing a homeomorphism $\wtu/G \xrightarrow{\cong} U$.

An \emph{orbifold} of dimension $d$ is a second countable Hausdorff topological space together with a maximal atlas of $d$-dimensional orbifold coordinate charts. See, for example, Definition 2.1 of \cite{GGKM-SSS} for the notion of compatibility of orbifold charts.
\end{definition}

\begin{remark}\label{rm:orbori}
While manifolds are always locally orientable, an orbifold is locally orientable if and only if the isotropy group at every point consists of orientation-preserving transformations, i.e., it is a subgroup of $SO(d)$. Global orientability of an orbifold further requires that the local charts can be compatibly oriented.  As in the manifold case, global orientability is equivalent to the existence of a nowhere vanishing continuous $d$-form.
\end{remark}
\begin{definition}\label{def:iso}
We define the isotropy type of $x$ as follows:   A chart $\cc$ about $x$ defines a smooth action of $G_{U}$ on $\widetilde U \subset \mathbb R^d$. Fix a lift $\widetilde x \in \widetilde{U}$ of $x$ and let $\Iso(\widetilde{x})$ be the isotropy subgroup of $G_U$ at $\widetilde{x}$.   The map $\gamma\mapsto  d\gamma_{\widetilde{x}}\in \End(T_{\widetilde{x}}\widetilde{U})$, defines an injective linear representation of $\Iso(\widetilde{x})$.    Every finite-dimensional linear representation of a compact Lie group is equivalent to an orthogonal representation, unique up to orthogonal equivalence.    Thus $\Iso(\widetilde{x})$ can be viewed as a subgroup of the orthogonal group $O(d)$, unique up to conjugacy.    The conjugacy class of $\Iso(\widetilde{x})$ in $O(d)$ is independent both of the choice of the lift $\widetilde{x}$ of $x$ in $\widetilde{U}$ and of the choice of chart $\cc$ and is called the \emph{isotropy type} of $x$.
\end{definition}

Next we recall the definition of the singular stratification.  See, for example, Section 2.1 of \cite{GGKM-SSS} for more detail.

\begin{definition}\label{def:singularstratification} 
Let $\orb$ be an orbifold of dimension $d$.
\begin{enumerate}
\item Every orbifold admits a stratification, as follows. We define an equivalence relation on $\orb$ by saying that two points in $\orb$ are \emph{isotropy equivalent} if they have the same isotropy type. The connected components of the isotropy equivalence classes of $\orb$ form a smooth stratification of $\orb$. When the corresponding isotropy type is non-trivial these strata are called \emph{singular strata}.  Note that in the literature, the requirement that a singular stratum be connected is sometimes dropped.
\item Let $N$ be a singular stratum in $\orb$ and let $\iiso(N)<O(d)$ denote a representative of its isotropy type.  We will refer to $\iiso(N)$ as the \emph{isotropy group} of the stratum.
\item\label{it:isomax-primary} We denote by $\isomax(N)$ the set of all elements $\gamma \in \iiso(N)$ such that the dimension of the $1$-eigenspace of $\gamma$ is equal to the dimension of $N$.  We say that a stratum $N$ is \emph{primary} if $\isomax(N)$ is non-empty.
\end{enumerate}
\end{definition}

\subsection{Differential forms, the Hodge Laplacian, and the Heat asymptotics for differential $p$-forms on orbifolds.}
Differential forms on orbifolds and the corresponding Hodge Laplacian on $p$-forms can be defined as in Section 2.2 of \cite{GGKM-SSS}. The Hodge Laplacian can be defined on the space of differential $p$-forms whose components are square-integrable, denoted by $L^2_p(\O),$ where it is essentially self-adjoint and has a discrete spectrum. We denote the closure of this operator by $\Delta^p.$ We apply the spectral theorem to this operator and denote its eigenvalues by $0\leq\lambda^{\p}_1 \leq \lambda^{\p}_2\leq\dots \to +\infty$, with associated smooth eigenforms, denoted by $(\varphi_i)_i$, which form an orthonormal basis of $ L_p^2(\O)$.

Recall if $\O$ has empty singular set then $\O$ is a closed Riemannian manifold, $M$, and the heat trace has a small-time asymptotic expansion \begin{equation}\label{manifoldasymp}
\sum_{j=0}^\infty\, e^{-\lambda_j^{\p} t} \sim_{t\to 0^+}
(4\pi t) ^{-d/2} \sum_{i=0}^{\infty} a^p_i(M) t^i,
\end{equation}
where the $a^p_j$, the so-called heat invariants, are integrals over $M$ of universal polynomials in the curvature and its covariant derivatives.
(See, for example, \cite{P70} and references therein.)

In Section 3 of \cite{GGKM-SSS}, the authors constructed the  heat kernel and the heat trace for the orbifold case. The universal expressions defining the $a_j^p$ make sense on Riemannian orbifolds as well as on manifolds.   We will use the same notation $a_j^p$ for the extension of the function $a_j^p$ to the class of all closed Riemannian orbifolds.

\begin{notarem}\label{Atrp} Let $\gamma\in O(d)$.   
\begin{enumerate}
\item We denote by $\tr_p(\gamma)$ the trace of the natural action of $\gamma$ on the space $\wedge^p(\R^d)$ of alternating $p$-tensors. 
\item Let $A_\gamma= \gamma_{|E_1(\gamma)^\perp}$, where $E_1(\gamma)<O(d)$ is the 1-eigenspace of $\gamma$.

Note that if $N$ is a singular stratum of an orbifold $\O$ and $\gamma\in \isomax(N)$, then $\dim(E_1(\gamma))=\dim(N)$ where $\dim(N)$ is the dimension of $N$ in $\O$.
\end{enumerate}

\end{notarem}

We recall the heat trace asymptotics and an explicit formula for some of the heat invariants using the previous notation.

\begin{thm}\label{thm.asympt}[Theorem 3.15,\cite{GGKM-SSS}] Let $\O$ be a closed $d$-dimensional Riemannian orbifold, let $p\in \{1,\dots, d\}$ and let $0\leq\lambda_1^{\p}\leq\lambda_2^{\p}\leq\dots \to +\infty$ be the spectrum of the Hodge Laplacian acting on smooth $p$-forms on $\O$.   The heat trace yields an asymptotic expansion as $t\to 0^+$ given by
\begin{equation}
\label{eq:traceO}
 \sum_{j=1}^{\infty}\,e^{-\lambda_j^{\p}t}\,\sim_{t\to 0^+}\, I_0^p(t)+\sum_{N\in PS(\O)} \, \frac{I_N^p(t)}{|\iso(N)|},
 \end{equation}
where $PS(\O)$ is the set of all primary singular $\O$-strata, $|\iso(N)|$ is the order of the isotropy group of $N$.  Here 
\begin{equation}\label{Ip0t}
 I_0^p(t)\coloneqq(4\pi t)^{-d/2}\sum_{k=0}^\infty\,a_k^p(\O) t^k   
\end{equation}
with $a_k^p(\O)$ given as above, and for $N\in PS(\O)$, 
\begin{equation}\label{IpNt}
 I_N^p(t):=(4\pi t)^{-\dim(N)/2}\sum_{k=0}^\infty\,b_k^p(N) t^k.   
\end{equation} 
  The coefficients $b_k^p(N)$ are of the form 
$$b_k^p(N)=\sum_{\gamma\in \isomax(N)}\, \int_N\,b_k^p(\gamma, x)\,dV(x) $$
where the $b_k^p$ are universal orthogonally invariant expressions in the germs of the Riemannian metric of $\O$ at $x$  and the action of $\gamma$.  

 In the notation of~\ref{Atrp}, we have
\begin{equation}\label{bzero}b_0^p(N)=\sum_{\gamma\in \isomax(N)}\, \frac{\tr_p(\gamma)}{\lvert \det(\Id_{\codim(N)}-A_\gamma)\rvert}.\end{equation}
\end{thm}

The asymptotic expansion~\eqref{eq:traceO} is of the form
\begin{align}
    \label{heatascoeffc}
(4\pi t)^{-d/2}\sum_{j=0}^\infty\, c^p_j(\mathcal O)\,t^{\frac{j}{2}}
\end{align}
with $c^p_j(\mathcal O)\in\R$.


To analyze the contributions from the $b_k^p(N)$ to the $c^p_j(\mathcal{O})$ coefficients, we recall the following notation from Section 2 of \cite{GGKM-SSS}.

\begin{notarem}\label{nota:Rstuff} 
For $\gamma \in O(d)$, we define the \emph{eigenvalue type} of $\gamma$ as follows:   Let $r$ be the dimension of the $(-1)$-eigenspace of $\gamma$; in particular, $r=0$ if $-1$ is not an eigenvalue.  Let $e^{\pm i \theta_j}$, $j=1,\dots, s$ be all the eigenvalues with non-trivial imaginary part, repeated according to multiplicity. Observe that the dimension of the $(+1)$-eigenspace is then $d-2s-r$. The expression $E(\theta_1, \theta_2, \dots, \theta_s;r)$ will be called the eigenvalue type of $\gamma$. When $r=0$, respectively $s=0$, we instead write $E(\theta_1, \theta_2, \dots, \theta_s;)$, respectively $E(;r)$.   

\end{notarem}

With this notation, we recall the following result,  Proposition 4.3 of \cite{GGKM-SSS}, which will be needed in the subsequent sections for results involving the $1$-spectrum.
The result follows as the $1$-trace is the trace of the matrix representation of $\gamma$ in $O(d)$ (as in Definition \ref{def:iso}).

\begin{prop}\label{b01general}[Proposition 4.3, \cite{GGKM-SSS}]  Let $N$ be a singular stratum of codimension $k$ in the $d$-dimensional closed orbifold $\O$. Suppose $\gamma \in\isomax(N)$ has eigenvalue type $E(\theta_1, \theta_2, \dots, \theta_s;r)$, using Notation~\ref{nota:Rstuff}.

Then
\begin{equation}\label{eq.b01}b_0^1(\gamma)=\left( d-k-r+\sum_{j=1}^s\,2\cos(\theta_j)\right) \left( 2^{-k}\prod_{j=1}^s\,\csc^2(\theta_j/2)\right) .\end{equation}
Here we use the convention that when $s=0$, we have $\prod_{j=1}^s\,\csc^2(\theta_j/2)=1$.
\end{prop}

\subsection{Krawtchouk polynomials and their role in heat invariants}\label{sec.computation}

The (binary) Krawtchouk polynomial of degree $p$ is defined as
\begin{equation}\label{krawtpoly}K^d_p(x)=\sum_{j=0}^p\,(-1)^j\,\binom{x}{j}\binom{d-x}{p-j}.\end{equation}

In some settings, the Krawtchouk polynomials are related to the heat invariants. We give the following example that will be used later to prove Theorem \ref{thm:krawt4}.

\begin{ex}\label{isotropy2} ~
Suppose $N$ is a singular stratum of codimension $k$ with isotropy group of order $2$.  The generator $\gamma$ of $\Iso(N)$ must have eigenvalue type $E(;k)$, and thus
\begin{equation}\label{Krawt_trace}\tr_p(\gamma)=\sum_{j=0}^p\,(-1)^j\,\binom{k}{j}\binom{d-k}{p-j}\end{equation}
with the understanding that $\binom{m}{n}=0$ when $n>m$.     
By Theorem~\ref{thm.asympt} equation~\eqref{bzero} and equation \eqref{Krawt_trace}, we have
\begin{equation}
\label{Krawt}b_0^p(N) =\frac{\vol(N)}{2^k}K_p^d(k).
\end{equation}
\end{ex}

In particular, in the previous example, $b_0^p(N)$ vanishes if and only if the codimension $k$ of $N$ is a zero of the Krawtchouk polynomial $K_p^d$. 

\begin{remark}\label{rm:zeros_of_K}
There is a large literature on the zeros of Krawtchouk polynomials (see, for example, \cite{CS90,KL96}).  We include a few elementary observations here:
\begin{enumerate}
\item $K^d_0(k)=1$ and $K^d_d(k) = (-1)^k$.
\item $K^d_1(k)=0$ if and only if $k=\frac{d}{2}$.
\item $K^d_2(k)=0$ if and only if the dimension $d=
n^2$ is a perfect square and $k=\frac{n(n\pm 1)}{2}$.  (For later use, we observe that if $4|n$, then both zeros $\frac{n(n\pm 1)}{2}$ are even.)
\item When $d$ is even and $p=\frac{d}{2}$, we have $K^d_p(k)=0$ for all odd $k$, see \cite[Remark 3.14]{MR02}.
\item When $d$ is even and $p$ is odd, we have $K^d_p(\frac{d}{2})=0$, see \cite[Remark 3.14]{MR02}.
\item It is remarked in \cite[Remark 3.14]{MR02} that $d=9$ is the lowest odd dimension for which some of the Krawtchouk polynomials $K_p^d(k)$ have an integral zero and that the next odd dimension with an integral zero is $d=17$.
\end{enumerate}
\end{remark}

\section{Negative Results}

We construct examples of orbifolds that are $p$-isospectral to manifolds (for various $p \geq 1$) and obtain some inverse spectral results for the $p$-spectrum alone. We also construct a family of mutually 1-isospectral 2-orbifolds with different geometric and topological properties.

\subsection{Isospectrality of manifolds and orbifolds for various $p$-spectra}\label{sec.isospec}

The article \cite{GR03} contains examples of orbifolds of even dimension $d=2m$ that are $m$-isospectral to manifolds. Here, we will construct examples of flat orbifolds that are $p$-isospectral to manifolds for various values of $p$.

\begin{nota}\label{nota:flat} Every closed flat orbifold or manifold is of the form $\O=\Sigma\bs \R^d$ where $\Sigma$ is a discrete subgroup of the Euclidean motion group $\R^d\rtimes O(d)$.  (Note that $\O$ is a manifold, i.e., $\Sigma$ acts freely on $\R^d$, if and only if $\Sigma$ is a Bieberbach group). The restriction of the projection $\R^d\rtimes O(d)\to O(d)$ to $\Sigma$ has finite image $F$ and kernel a lattice $\Lambda$ of rank $d$.   We will refer to $\Lambda$ as the translation lattice of $\Sigma$. The group $F$ is the holonomy group of $\O$.   For each $\gamma\in F$, there exists $a=a(\gamma)\in \R^d$, unique modulo $\Lambda$, such that $\gamma\circ L_a\in \Sigma$, where $L_a$ denotes translation by $a$.   Let $\Lambda^*$ denote the lattice dual to $\Lambda$.  For $\mu\geq 0$ and $\gamma\in F$, set
$$e_{\mu,\Sigma}(\gamma)=\sum_{v\in \Lambda^*, \|v\|=\mu,\gamma(v)=v}e^{2\pi i v\cdot a}.$$

\end{nota}

In the notation of Notation~\ref{nota:flat}, let $T$ be the torus $\Lambda\bs\R^d$.    Let $\eta =\sum_J \, f_J\,dx^J$ be the pullback to $\R^n$ of a $p$-form on $\O$.  (Here $J$ varies over all multi-indices $1\leq j_1<\dots<j_p\leq d$, and the functions $f_J$ are $\Sigma$-invariant.)  We have
$$\Delta^p(\eta)=\sum_J \,\Delta^0( f_J)\,dx^J.$$
Thus every element of the $p$-spectrum of $\O$ occurs as an eigenvalue in the 0-spectrum of the torus $T$; i.e., it is of the form $4\pi^2\|\mu\|^2$ for some $\mu$ in the dual lattice $\Lambda^*$ of $\Lambda$.   The following result of Miatello and Rossetti, while originally stated in the context of flat manifolds, is also valid for orbifolds.

\begin{prop}[{\cite[Theorem 3.1]{MR01}}]
\label{flatspeccomp}
We use the notation of Notation~\ref{nota:flat}.
\begin{enumerate}
\item For $\mu\geq 0$, the multiplicity $m_{p,\mu}(\Sigma)$ of $\mu$ in the $p$-spectrum of $\O=\Sigma\bs\R^d$ is given by
$$m_{p,\mu}(\Sigma)=\frac{1}{|F|}\sum_{\gamma\in F}\, \tr_p(\gamma) e_{\mu,\Sigma}(\gamma)=\frac{1}{|F|}\binom{d}{p}+\frac{1}{|F|}\sum_{1\neq\gamma\in F}\, \tr_p(\gamma) e_{\mu,\Sigma}(\gamma).$$
\item \label{it:flatspeccomp-isospectral} Thus if $\O'=\Sigma'\bs \R^d$, where $\Sigma'$ has the same translation lattice $\Lambda'=\Lambda$, and if there is a bijection $\gamma\mapsto\gamma'$ between the holonomy groups of $\O$ and $\O'$ such that $$\tr_p(\gamma) e_{\mu,\Sigma}(\gamma)=\tr_p(\gamma') e_{\mu,\Sigma'}(\gamma')$$
for all $\gamma\in F$, then $\O$ and $\O'$ are $p$-isospectral.

\end{enumerate}
\end{prop}

\begin{cor}\label{isosp condition}
Suppose $\O=\Sigma\bs \R^d$ and $\O'=\Sigma'\bs \R^d$ where $\Sigma$ and $\Sigma'$ have the same translation lattice.   If $|F|=|F'|$ and if $\tr_p(\gamma)=0=\tr_p(\gamma')$ for all $1\neq \gamma\in F$ and $1\neq \gamma'\in F'$, then $\O$ and $\O'$ are $p$-isospectral.

\end{cor}

\begin{nota}\label{nota.okmk}
Given $d\geq 2$ and $k\in \{1, \dots, d-1\}$, let $\gamma_k:\R^{d}\to\R^{d}$ be given by
$$\gamma_k(x_1,\dots, x_d)=(-x_1,\dots, -x_k, x_{k+1},\dots, x_d),$$
let $a\in (\frac{1}{2}\Z)^d$ with at least one of the last $d-k$ entries of $a$ equal to $\frac{1}{2}$, and let $\rho_k=\gamma_k\circ L_a$.   Let $\Sigma_k$, respectively $\Sigma'_k$ be the discrete subgroup of the Euclidean motion group $\R^d\rtimes O(d)$ generated by $\Z^d$ together with $\gamma_k$, respectively $\rho_k$.   Set $\O_k:=\Sigma\bs \R^d$ and $M_k:= \Sigma'\bs \R^d$.  
\end{nota}

\begin{prop}\label{it:flat-involution}~We use Notation~\ref{nota.okmk}.  
\begin{enumerate}
\item Suppose that $K_p^d(k)=0$, where $K_p^d$ is the Krawtchouk polynomial given by Equation~\eqref{krawtpoly}.  Then $\O_k$ and $M_k$ are $p$-isospectral.
\item
Moreover, if $k'\in \{1, \dots, d-1\}$ is another zero of $K_p^d$, then the collection $\{M_k, M_{k'}, \O_k, \O_{k'}\}$ are all mutually $p$-isospectral.
\item $M_k$ is a manifold for all $k$, whereas the singular set of $\O_k$ consists precisely of $2^k$ singular strata each of which has codimension $k$.
\item $\O_k$ and $M_k$ are orientable if and only if $k$ is even.

\end{enumerate}
\end{prop}

\begin{proof}
(1 and 2) In the notation of Notation~\ref{nota:flat}, the holonomy groups of both $\O_k$ and $M_k$ have order 2 with generator $\gamma_k$.   Since $\tr_p(\gamma_k)=K^d_p(k)=0$ (as seen in Example~\ref{isotropy2}), Corollary~\ref{isosp condition} implies that $\O_k$ and $M_k$ are $p$-isospectral .  Moreover, the corollary also implies that their $p$-spectra are independent of the choice of the zero $k$ of $K^d_p$.

(3) The torus $T=\Z^d\bs  \R^d$ is a two-fold cover  of each of $\O_k$ and $M_k$.  The map $\rho_k$ induces a fixed-point free involution of $T$, so $M_k$ is a manifold.    In contrast, the involution of $T$ induced by $\gamma_k$ fixes all points each of whose first $k$ coordinates lie in $\{0, \frac{1}{2}\}$. Thus $\O_k$ contains $2^k$ singular strata each of codimension $k$. 

(4) is immediate.

\end{proof}
The examples below follow from Proposition~\ref{it:flat-involution}.
\begin{ex}\label{cor.ornon} In the notation of Notation~\ref{nota:flat}, 
in every even dimension $d$, there exists a Bieberbach manifold $M$ and a flat $d$-dimensional orbifold with singular set of codimension $\frac{d}{2}$ such that $M$ and $\O$ are $p$-isospectral for \emph{all odd} $p$.  (See Remark~\ref{rm:zeros_of_K} part 5.)

\end{ex}

\begin{ex} By Equation~\eqref{krawtpoly}, one easily sees that if 
$k$ is a zero of $K^d_p$, then so is $d-k$.  Moreover, observe that $\O_k$ and $M_k$ are orientable if and only if $k$ is even.  In particular, if $d$ is odd and $K_p^d(k)=0$, then the collection of $p$-isospectral orbifolds and manifolds $\{M_k,\O_k,M_{d-k}, \O_{d-k}\}$ contains an orientable and a non-orientable manifold and an orientable and a non-orientable orbifold.  Moreover, the codimensions of the singular sets of the two orbifolds have different parity.   

For a specific example, we can take $d=n^2$ with $n\geq 3$, $p=2$ and $k=\frac{n(n+1)}{2}$ as in Remark~\ref{rm:zeros_of_K}, part 3.
This example 
was motivated by Example 4.3 in Miatello-Rossetti \cite{MR01}, where the case $n=3$ was considered in the manifold case.  We will use this example to prove sharpness in a result in the next section.
\end{ex}

\subsection{Negative inverse spectral results for the middle degree}\label{sec.dim1exa}

For manifolds and orbifolds of even dimension $d$, the $\frac{d}{2}$-spectrum contains less information than the other Hodge spectra due to Hodge duality: in any degree $p$, the Hodge Laplacian leaves invariant the subspace of  exact forms and the subspace of co-exact forms.   In the middle degree $p=\frac{d}{2}$, the spectra on these two subspaces are identical.  See \cite{GR03} for many examples illustrating the weakness of the middle degree spectrum.  

\begin{ex}\label{ex:d/2}Proposition~\ref{it:flat-involution} together with Remark~\ref{rm:zeros_of_K} part 4 yield $\frac{d}{2}$-isospectral flat manifolds and orbifolds that are $p$-isospectral for all odd $p$.   \cite[Theorem 3.2]{GR03} proves the isospectrality of these manifolds and orbifolds by different methods.  In the special case that $d=2$, the resulting manifold is a Klein bottle, while the orbifold is a cylinder, and these were also shown to be 1-isospectral to a M\"obius strip.  In \cite{GR03}, it was also asserted that these surfaces are $1$-isospectral to a flat $3$-pillow (i.e., a flat Riemannian orbifold whose underlying space is a sphere with three cone points); however as shown in the subsequent errata \cite{GR21}, that assertion was incorrect.
\end{ex}

The example below describes a similar family of mutually $1$-isospectral $2$-dimensional orbifolds, demonstrating further the weakness of the middle-degree spectrum as a topological invariant. The $1$-isospectrality of the first and the fifth orbifolds is particularly striking as the first can be viewed as a square with absolute boundary conditions, while the second is topologically a sphere with three cone points. We also see that the $1$-spectrum does not detect the maximum order of isotropy present in an orbifold: the second and third orbifolds have maximum isotropy of order two in contrast to maximum order four in the other three. 

\begin{ex}\label{ex:isospecorb}

We construct a family of five mutually 1-isospectral 2-dimensional flat orbifolds.  We use  Notation~\ref{nota:flat}.   All the orbifolds will have the same translation lattice $\Lambda=\Z^2$, and their holonomy groups will have order four.  To define each orbifold, we will specify both its holonomy group $F < O(2)$ and a choice of $a(\gamma)$ for each $\gamma\in F$.   The orbifold is then given by $\O:=\Sigma\bs\R^2$ where $\Sigma < \R^2\rtimes O(2)$ is generated by $\Lambda \cup \{\gamma\circ L_{a(\gamma)}:\gamma \in F\}$. In each example, $F$ contains $\Id$, $-\Id$, and two elements with 1-trace zero. One can easily verify the mutual isospectrality of the five orbifolds using Proposition~\ref{flatspeccomp}, part~\ref{it:flatspeccomp-isospectral}. The symbolic expression at the end of each example gives the Conway notation for the orbifold diffeomorphism class to which $\O$ belongs.

\begin{enumerate}
\item\label{exa:flatsquare} Let $F<O(2)$ be the Klein 4-group generated by the reflections $\gamma_1$ and $\gamma_2$ across the $x$ and $y$-axes, respectively. Take $a(\gamma_1)$ and $a(\gamma_2)$ both trivial. The underlying space of the corresponding orbifold $\O$ is a square, and the $1$-spectrum of $\O$ as an orbifold coincides with the $1$-spectrum of the square with absolute boundary conditions.  As an orbifold, the four edges are reflectors and the four corners are order $2$ dihedral points. Thus $\O$ is of class $*2222$. 
\item Let $F<O(2)$ be as in item~\ref{exa:flatsquare} now with $a(\gamma_1)= (\tfrac{1}{2},0)$ and $a(\gamma_2)$ trivial. The underlying space of $\O$ is a disk, and the singular set consists of two cone points of order $2$ and a mirror edge along the boundary of the disk. Thus $\O$ is of class $22*$.
\item Let $F<O(2)$ be as in item~\ref{exa:flatsquare} now with $a(\gamma_1)= (\tfrac{1}{2},0)$ and $a(\gamma_2) = (0,\tfrac{1}{2})$. The underlying space of $\O$ is a projective plane, and the singular set consists of two cone points of order $2$. Thus $\O$ is of class $22\times$.
\item Let $F<O(2)$ be the Klein 4-group generated by the reflections $\gamma_3$ and $\gamma_4$ across the lines $y=x$ and $y=-x$, respectively. Take $a(\gamma_3)$ and $a(\gamma_4)$ both trivial. The underlying space of $\O$ is a disk, and the singular set consists of two order $2$ dihedral points on the boundary of the disk connected by two reflector edges, along with an interior order $2$ cone point. Thus $\O$ is of class $2*22$. 
\item  Let $F<O(2)$ be the cyclic group of order $4$ generated by rotation $\gamma$ through angle $\frac{\pi}{2}$ about the origin, each element acting without precomposition by a translation. The underlying space of $\O$ is a sphere, and the singular set consists of two cone points of order $4$ and one cone point of order $2$. Thus $\O$ is of class $244$.
\end{enumerate}
\end{ex}

\begin{remark}\label{not 0-isosp} The five orbifolds in Example~\ref{ex:isospecorb} are all mutually distinguishable by their 0-spectra.  This follows, for example, by comparing the heat invariants for these orbifolds using \cite[Table 1]{DGGW08} and the fact that the lengths of the mirror loci in the first, second and fourth orbifolds are mutually distinct.

\end{remark}

\section{Positive Results}

In this section we obtain positive results regarding our central question of isospectrality between Riemannian manifolds and orbifolds. We first consider the $p$-spectrum for $p \geq 0$ and then focus our attention on the case where $p=1$.

\subsection{Positive inverse spectral results for the $p$-spectra}\label{sec: common cover}

We continue to assume that $\O$ is a $d$-dimensional closed Riemannian orbifold.

\begin{nota}\label{parity}
We will say that a singular stratum of $\O$ has \emph{positive, respectively negative, parity} if its codimension in $\O$ is even, respectively odd.  For $\epsilon\in\{\pm\}$, let $k_\epsilon$ be the minimum codimension of the primary singular strata (if any) of $\O$ of parity $\epsilon$, and let $S_\epsilon(\O)$ be the collection of all primary strata of codimension $k_\epsilon$.  Set
$$B_\epsilon^p(\O) =\sum_{N\in S_\epsilon(\O)}\,\frac{1}{|\Iso(N)|}\,b_0^p(N).$$
If there are no singular strata of parity $\epsilon$, then we set $B_\epsilon^p(\O)=0.$
\end{nota}

\begin{remark}\label{rem:p=0}
In case $p=0$, one has $b_0^0(N) > 0$ for every primary singular stratum.   Thus  the condition $B_-^0(\O)\neq 0$, respectively $B_+^0(\O)\neq 0$, is equivalent to the condition that $\O$ contains at least one primary singular stratum of odd, respectively even, codimension. In the following proposition we obtain weak analogues of results for functions from \cite{GR03} and its errata \cite{GR21} and \cite[Theorem 1.2]{S10}. 

\begin{prop}\label{prop: odd par}~
\begin{enumerate}
\item \label{it:oddpar-1} $B_-^p(\O)$ is an invariant of the $p$-spectrum.   In particular, if $B_-^p(\O)\neq 0$, then $\O$ cannot be $p$-isospectral to a Riemannian manifold.
\item \label{it:oddpar-2}  If $B_+^p(\O)\neq 0$, then $\O$ cannot be $p$-isospectral to any closed Riemannian manifold
$M$ such that $M$ and $\O$ have isometric (possibly infinite) homogeneous Riemannian covers.  If $B_+^p(\O)\neq 0$ and $\O$ is good, then $\O$ cannot be $p$-isospectral to any closed Riemannian manifold $M$ such that $M$ and $\O$ have finite $p$-isospectral Riemannian covers.
\end{enumerate}
\end{prop}

\begin{proof}~
To prove part~\ref{it:oddpar-1} observe that because $k_-$ is the minimal codimension of the odd primary singular strata, we have that $c_{k_-}^p(\O)=(4\pi)^{k_-/2}B_-^p(\O)$ in the notation of Equation~\eqref{heatascoeffc}, and thus $B_-^p(\O)$ is a spectral invariant.  The second statement of part~\ref{it:oddpar-1} is immediate.

The proof of part~\ref{it:oddpar-2} is similar to the proofs of the analogous statements for the 0-spectrum cited in the introduction and above.  We give a summary.  It suffices to show that if $M$ and $\O$ are isospectral, then each of the two conditions on $M$ and $\O$ implies that $a_j^p(\O)=a_j^p(M)$ for all $j$.   (See equation \eqref{Ip0t}.)   Writing $k_+=2\ell$, one then notes that $c^p_{k_+}(\O) = a_\ell^p(\O)+(4\pi)^\ell B_+^p(\O)$ whereas the corresponding coefficient for $M$ is $a_\ell^p(M)$, contradicting the $p$-isospectrality of $\O$ and $M$.

First consider the case that there exist finite Riemannian covers $M^*$ of $\O$ and $M^{**}$ of $M$ that are $p$-isospectral.  Suppose that $M$ and $\O$ are $p$-isospectral.  Since the $p$-spectrum determines the volume, we have $\vol(M)=\vol(\O)$ and also $\vol(M^{**})=\vol(M^*)$.   Thus the two coverings are of the same order $r$.    Using equation \eqref{Ip0t}, we then have
$$a_j^p(\O) =\frac{1}{r}a_j^p(M^*)=\frac{1}{r}a_j^p(M^{**})= a_j^p(M)$$
for all $j$.

Next suppose that $\O$ and $M$ have a common homogeneous Riemannian cover $\tilde{M}$.  We use the fact that the invariants $a_j^p(M)$ and $a_j^p(\O)$ are integrals over $M$ and $\O$, respectively, of functions $U_j^p(M;\cdot)$ and $U_j^p(\O,\cdot)$ that satisfy locality and universality. (See \cite[Notation 3.13]{GGKM-SSS}.) Homogeneity of the cover, along with the locality and universality properties, implies that $U_j^p(M;\cdot)$ and $U_j^p(\O,\cdot)$ are constant functions with the same constant value $U_j$.  If $M$ and $\O$ are isospectral, then again they have the same volume $V$, so $a_j^p(M)=U_jV=a_j^p(\O)$.
\qedhere

\end{proof}

\end{remark}

\begin{thm}\label{thm:krawt4} Denote by ${\mathcal Orb}^d_k$ the class of all closed $d$-dimensional Riemannian orbifolds with singular set of codimension $k$.    Assume that $k$ is odd.  Let $p\in \{0,1,2,\dots , d\}$ and let $K_p^d$ be the Krawtchouk polynomial given by Equation~\eqref{krawtpoly}. If $K_p^d(k)\neq 0$, then the $p$-spectrum determines the $(d-k)$-dimensional volume of the singular set of elements of ${\mathcal Orb}^d_k$.   In particular, the $p$-spectrum distinguishes elements of orbifolds in ${\mathcal Orb}^d_k$ from Riemannian manifolds.
\end{thm}

\begin{remark}\label{rm:must-be-odd}
The condition requiring that the parity of the codimension of the singular set is odd in Theorem \ref{thm:krawt4} cannot be removed. In \cite{RSW08}, Example 3.3, a pair of flat three-dimensional orbifolds were described, $\O_1$ and $\O_2$, which are $0$-isospectral but not $1$-isospectral. In this example, the singular sets of both $\O_1$ and $\O_2$ are of codimension 2. Further, as noted in the above paper, the singular set of $\O_2$ has volume equal to 12. On the other hand, one can check that the volume of the singular set of $\O_1$ is 6.
\end{remark}

\begin{proof}[Proof of Theorem \ref{thm:krawt4}] Let $N \subset \orb \in {\mathcal Orb}^d_k$ be a singular stratum of codimension $k$.  View $\Iso(N)$ as a subgroup of the orthogonal group $O(d)$.  All non-trivial elements of $\Iso(N)$ must lie in $\isomax(N)$ and must have the same 1-eigenspace; otherwise there would exist a stratum of smaller codimension.  Thus  $\Iso(N)$ must be of the form
$\Iso(N)=\Gamma\times \{\Id_{d-k}\}$ where $\Gamma$ is a subgroup of the orthogonal group $O(k)$ that acts freely on the sphere $S^{k-1}$.   Since $k$ is odd and the only finite group acting freely on even-dimensional spheres has order 2, $\Gamma$ and thus also $\Iso(N)$ must have order 2.

Using Equation~\eqref{Krawt} and Notation~\ref{parity} we have,
\begin{equation*} B_-^p(\orb)=\frac{K_p^d(k)}{2^{k+1}}\sum_{N \in S_-(\orb)}\,\vol(N).
\end{equation*}
As shown in Proposition~\ref{prop: odd par} part~\ref{it:oddpar-1}, $B_-^p(\O)$ is a spectral invariant, thus the proof is complete.
\end{proof}

Thus for example, Theorem~\ref{thm:krawt4} along with Remark~\ref{rm:zeros_of_K} imply the following.

\begin{cor}~
\begin{enumerate}
    \item Assume $d$ is not a perfect square. Then no closed Riemannian $d$-orbifold with singular set of odd codimension can be 2-isospectral to a Riemannian manifold (see Remark \ref{rm:zeros_of_K}, part 3). The same statement holds if $d$ is of the form $d=4m^2$ for some integer $m$. 
    \item Assume that $d < 9$ is odd. Then no closed Riemannian $d$-orbifold with singular set of odd codimension can be $p$-isospectral to a Riemannian manifold for $p \in \{0, 1, 2, \dots, d-1\}$ (see Remark \ref{rm:zeros_of_K}, part 6).
\end{enumerate}
\end{cor}

\subsection{Positive inverse spectral results for the $1$-spectrum}\label{ss:positive1form}
We next focus on the $1$-spectrum.
We treat this case in more depth as the $1$-trace is the trace of the matrix representation of the isometries in $O(d)$ and this makes the corresponding calculations more tractable.

We begin with a theorem about what the $1$-spectrum reveals about orbifolds with a relatively large dimensional singular set.

\begin{thm}\label{thm:halfdim}
Let $\orb$ be a closed Riemannian orbifold of dimension $d$.
\begin{enumerate}

\item \label{it:halfdim-primary}  If $\orb$ contains at least one primary singular stratum of odd codimension $k<\frac{d}{2}$, then $\orb$ cannot be 1-isospectral to any closed Riemannian manifold.

\item \label{it:halfdim-covers} If $\orb$ contains at least one singular stratum of codimension $k<\frac{d}{2}$, then $\O$ cannot be $1$-isospectral to any closed Riemannian manifold
$M$ such that $M$ and $\O$ have isometric infinite homogeneous Riemannian covers.  If in addition $\orb$ is good, then $\O$ cannot be $1$-isospectral to any closed Riemannian manifold $M$ such that $M$ and $\O$ have finite $1$-isospectral Riemannian covers.

\item \label{it:halfdim-deven} For $d$ even, both statements remain true when $k=\frac{d}{2}$ provided at least one stratum of codimension $k$ has isotropy group of order at least three.

\item \label{lem:isodeven} An element of ${\mathcal Orb}^d_k$ with $k$ odd can be 1-isospectral to a Riemannian manifold only if $d$ is even and $k=\frac{d}{2}$.

\end{enumerate}

\end{thm}

\begin{remark}\label{rem: sharp}

The proof of Proposition~\ref{it:flat-involution} shows that the hypothesis on isotropy order in Theorem~\ref{thm:halfdim} part~\ref{it:halfdim-deven} cannot be removed.

In Theorem~\ref{thm:halfdim}, part~\ref{it:halfdim-primary} and the first part of part~\ref{it:halfdim-covers}, the condition $k = \frac{d}{2}$ is sharp:  Example~\ref{cor.ornon} shows that the conclusion fails when $k = \frac{d}{2}$ with $d$ even (see Remark~\ref{rm:zeros_of_K}, part 5).

\end{remark}

\begin{proof}[Proof of Theorem~\ref{thm:halfdim}]  Parts \ref{it:halfdim-primary} and \ref{it:halfdim-covers} follow from Proposition \ref{prop: odd par} along with the following two observations:
\begin{enumerate}
\item Since any singular stratum of minimum codimension is necessarily primary,  the hypothesis of part~\ref{it:halfdim-covers} implies that $k_\epsilon <\frac{d}{2}$ for at least one $\epsilon\in\{\pm\}$, while the hypothesis of part~\ref{it:halfdim-primary} says that $k_-<\frac{d}{2}$.

\item If $N$ is any primary singular stratum of codimension $k$ and if $\gamma\in\isomax(N)$, then 1 is an eigenvalue of $\gamma$ with multiplicity $d-k$ while all other eigenvalues lie in $[-1,1)$.  Thus if $k<\frac{d}{2}$, then $\tr(\gamma)>0$ for all $\gamma\in\isomax(N)$, so $b_0^1(N) > 0$.   In particular, if $k_\epsilon(\orb)<\frac{d}{2}$, then $B_\epsilon(\orb)>0$.
\end{enumerate}

To prove part~\ref{it:halfdim-deven},
suppose that $k_\epsilon=\frac{d}{2}$ and let $N$ be a primary stratum of codimension $k=\frac{d}{2}$.   Every $\gamma\in \isomax(N)$ has eigenvalues $1$ with multiplicity $\frac{d}{2}$ while all other eigenvalues are either -1, or occur in conjugate pairs with real parts less than 1, so $\tr(\gamma)\geq 0$ with equality if and only if $\gamma^2= \Id$. If $N$ has higher order isotropy, then some $\gamma\in \isomax(N)$ has at least two eigenvalues with real part strictly less than one, and thus $b_0^1(N)>0$. Hence the presence of any primary strata with isotropy order greater than 2 implies that $B_\epsilon(\orb)>0$, so we can again apply Proposition \ref{prop: odd par}.

Part \ref{lem:isodeven} follows from Theorem~\ref{thm:krawt4} along with Remark~\ref{rm:zeros_of_K}, part 2.

\end{proof}



In \cite{GGKM-SSS}, we prove that the combination of the 0-spectrum and the 1-spectrum can distinguish closed orbifolds with singular sets of codimension at most 3 from smooth, closed Riemannian manifolds. As mentioned in the introduction, a natural question is whether the $1$-spectrum alone can distinguish orbifolds from manifolds? In the special case where $d=6$, we have that the heat invariant $a_1^1(\ast) = (d-6) a_1^0(\ast) = 0$ where $\ast$ is $M$ or $\orb$ (see, for example, \cite{P70}).
This observation motivates the investigation of whether, in dimension 6, it is possible for the 1-spectrum to distinguish orbifolds with singular sets of codimension 2 from Riemannian manifolds. With this in mind, we obtain the following result as an application of Theorem \ref{thm:halfdim} and part of the strategy of the proof of Theorem 1.1 in \cite{GGKM-SSS}.
\begin{thm}\label{thm:dim6}
Let $\orb$ be a $6$-dimensional, closed orbifold.
    \begin{enumerate}
    \item If $\orb$ has singular sets of codimension $\leq 2$, then the 1-spectrum distinguishes $\orb$ from smooth, closed, $6$-dimensional Riemannian manifolds. \label{it:app61}
    \item If $\orb$ contains at least one stratum of codimension 3 with isotropy group of order at least 3, then $\orb$ cannot be 1-isospectral to any smooth, closed, $6$-dimensional Riemannian manifold. \label{it:app62}
\end{enumerate}
\end{thm}

\begin{proof}

The second statement follows immediately from Theorem \ref{thm:halfdim} part \ref{it:halfdim-deven}.

To prove the first statement, we have from Theorem \ref{thm:halfdim} part \ref{it:halfdim-primary} that if $\orb$ contains at least one primary singular stratum of codimension 1, then $\orb$ cannot be 1-isospectral to any closed Riemannian manifold.

It remains to treat the case where $\orb$ has strata of codimension 2. Let $\mathcal{S}_2(\orb)$ denote the collection of strata of codimension 2.
Following the proof of Theorem 1.1 of \cite{GGKM-SSS}, if $N$ is a stratum of $\orb$ of codimension 2 then, as codimension 1 strata have already been considered above, $N$ must have cyclic isotropy group of order $m$. For $d=6$, the formula for $b_0^1(N)$ computed in \cite{GGKM-SSS} reads
\begin{equation*}
    b_0^1(N) = \left(\frac{m^2 -1}{3} + \frac{m^2-6m+5}{6}\right) \vol(N)
    = \frac{(m-1)^2}{2} \vol(N).
\end{equation*}
Hence $b_0^1(N) > 0$.

Now, since $a_1^1(\ast) = (d-6) a_1^0(\ast) = 0$, where $\ast$ is $M$ or $\orb$, the term of order $t^{-2}$
in the small-time asymptotic expansion for the heat trace for 1-forms on $\orb$ has coefficient
\begin{equation*}
    \sum_{N \in \mathcal{S}_2(\orb)} b_0^1(N) > 0,
\end{equation*}
but there is no such term in the small-time asymptotic expansion for the heat trace for 1-forms on $M$ as $a_1^1(M) = 0$.
\end{proof}

\bibliographystyle{alpha}
\bibliography{colibrisbibp2.bib}

\newcommand{\etalchar}[1]{$^{#1}$}
\def\cprime{$'$}
\begin{thebibliography}{DGGW08}

\bibitem[CS90]{CS90}
Laura Chihara and Dennis Stanton.
\newblock Zeros of generalized {K}rawtchouk polynomials.
\newblock {\em J. Approx. Theory}, 60(1):43--57, 1990.

\bibitem[DGGW08]{DGGW08}
Emily~B. Dryden, Carolyn~S. Gordon, Sarah~J. Greenwald, and David~L. Webb.
\newblock Asymptotic expansion of the heat kernel for orbifolds.
\newblock {\em Michigan Math. J.}, 56(1):205--238, 2008.

\bibitem[GGK{\etalchar{+}}23]{GGKM-SSS}
Katie Gittins, Carolyn Gordon, Magda Khalile, Ingrid~Membrillo Solis, Mary
  Sandoval, and Elizabeth Stanhope.
\newblock {Do the Hodge Spectra Distinguish Orbifolds from Manifolds? Part 1}.
\newblock {\em Michigan Mathematical Journal}, pages 1 -- 28, 2023.

\bibitem[GR03]{GR03}
Carolyn~S. Gordon and Juan~Pablo Rossetti.
\newblock Boundary volume and length spectra of {R}iemannian manifolds: what
  the middle degree {H}odge spectrum doesn't reveal.
\newblock {\em Ann. Inst. Fourier (Grenoble)}, 53(7):2297--2314, 2003.

\bibitem[GR21]{GR21}
Carolyn~S. Gordon and Juan~Pablo Rossetti.
\newblock Correction to ``{B}oundary volume and length spectra of {R}iemannian
  manifolds: what the middle degree {H}odge spectrum doesn't reveal''.
\newblock {\em Annales de l'Institut Fourier}, 71(6):2647--2648, 2021.

\bibitem[KL96]{KL96}
Ilia Krasikov and Simon Litsyn.
\newblock On integral zeros of {K}rawtchouk polynomials.
\newblock {\em J. Combin. Theory Ser. A}, 74(1):71--99, 1996.

\bibitem[MR01]{MR01}
R.~J. Miatello and J.~P. Rossetti.
\newblock Flat manifolds isospectral on {$p$}-forms.
\newblock {\em J. Geom. Anal.}, 11(4):649--667, 2001.

\bibitem[MR02]{MR02}
R.~J. Miatello and J.~P. Rossetti.
\newblock Comparison of twisted {$p$}-form spectra for flat manifolds with
  diagonal holonomy.
\newblock {\em Ann. Global Anal. Geom.}, 21(4):341--376, 2002.

\bibitem[Pat70]{P70}
V.~K. Patodi.
\newblock Curvature and the fundamental solution of the heat operator.
\newblock {\em J. Indian Math. Soc.}, 34(3-4):269--285 (1971), 1970.

\bibitem[RS20]{RS20}
Sean Richardson and Elizabeth Stanhope.
\newblock You can hear the local orientability of an orbifold.
\newblock {\em Differential Geom. Appl.}, 68:101577, 7, 2020.

\bibitem[RSW08]{RSW08}
Juan~Pablo Rossetti, Dorothee Schueth, and Martin Weilandt.
\newblock Isospectral orbifolds with different maximal isotropy orders.
\newblock {\em Ann. Global Anal. Geom.}, 34(4):351--366, 2008.

\bibitem[Sut10]{S10}
Craig~J. Sutton.
\newblock Equivariant isospectrality and {S}unada's method.
\newblock {\em Arch. Math. (Basel)}, 95(1):75--85, 2010.

\end{thebibliography}

\end{document}